\documentclass[12pt]{amsart}
\usepackage{amssymb,latexsym,amsmath,amsfonts,amsthm,amscd,graphicx,url,color,hyperref}
\usepackage{stmaryrd}

\theoremstyle{plain}

\newtheorem{thm}{Theorem}

\newtheorem{cor}[thm]{Corollary}

\theoremstyle{definition}

\title{Bounds for invariants of numerical semigroups and Wilf's Conjecture}

\author{Marco D'Anna}

\address[Marco D'Anna]{Dipartimento di Matematica e Informatica, \ Universit\`a di Catania, \  Viale Andrea Doria 6, 
	95125 Catania,Italy}

\email{mdanna@dmi.unict.it}

\author{Alessio Moscariello}

\address[Alessio Moscariello]{Dipartimento di Matematica e Informatica, \ Universit\`a di Catania, \  Viale Andrea Doria 6, 
	95125 Catania,Italy}

\email{alemoscariello@hotmail.it}



\begin{document}

\maketitle

\begin{abstract}
	Given coprime positive integers $g_1 < \ldots < g_e$, the Frobenius number $F=F(g_1,\ldots,g_e)$ is the largest integer not representable as a linear combination of $g_1,\ldots,g_e$ with non-negative integer coefficients. Let $n$ denote the number of all representable non-negative integers less than $F$; Wilf conjectured that $F+1 \le e n$. We provide bounds for $g_1$ and for the type of the numerical semigroup $S=\langle g_1,\ldots,g_e \rangle$ in function of $e$ and $n$, and use these bounds to prove that $F+1 \le q e n$, where $q= \left \lceil \frac{F+1}{g_1} \right \rceil$, and $F+1 \le e n^2$. Finally, we give an alternative, simpler proof for the Wilf conjecture if the numerical semigroup $S=\langle g_1,\ldots,g_e \rangle$ is almost-symmetric.

\end{abstract}

\section{Introduction}

The classical money-changing problem consists of finding what sums of money can be changed, using $e$ different denominations of coins $2 \le g_1 < \ldots < g_e$. Assuming, without loss of generality, that $\gcd(g_1,\ldots,g_e)=1$, it is well-known that only a finite number of sums cannot be changed, and there exists a maximum integer $F=F(g_1,\ldots,g_e)$ which cannot be represented as a linear combination of the \emph{generators}  $g_1,\ldots,g_e$, with coefficients in the set of natural numbers $\mathbb{N}$.

Determining this maximum integer $F$, called the Frobenius number, is the subject of the Diophantine Frobenius Problem. This challenging problem has been extensively studied over the past decades, and presents applications in several areas of mathematics, including Commutative Algebra, Combinatorics and Coding Theory (see \cite{RA} for a monograph on this problem). Nonetheless, as of today there is an exact solution only for the special case $e=2$, where Sylvester showed that $F=g_1g_2-g_1-g_2$. In the general case, it is known that no polynomial formula for $F$ in function of $g_1,\ldots, g_e$ can exist (cf. \cite{C}), and presently the literature is mostly focused on finding algorithms and bounds for $F$.

In 1978, H.S. Wilf proposed an upper bound for the Frobenius number $F$ (cf. \cite{W}), namely
\begin{equation}
	F+1 \le en,
\end{equation}
where $n$ is the number of solutions of the money-changing problem for $g_1,\ldots,g_e$ less than $F$ (actually, Wilf's original question was a lower bound for $e$, but we choose this equivalent and simpler formulation of his conjecture). 

This problem, now known as the Wilf Conjecture, has been considered by several authors; however, only special cases have been solved. For instance, it is known that the Conjecture is true in the following cases: $e \le 3$ (cf. \cite{FGH}), $|\mathbb{N} \setminus S| \le 65$ (where $S$ denotes the numerical semigroup generated by $g_1, \dots, g_e$; cf. \cite{BAR}), $e \ge \frac{g_1}{3}$ (cf. \cite{E2}), when $g_1$ is large enough and its prime factors are not smaller than $\left \lceil \frac{g_1}{e} \right \rceil$ (cf. \cite{MS}), if $F+1 \le 3g_1$ (cf. \cite{E1}). Most notably, the last case, due to Eliahou, coupled with a previous result by Zhai (cf. \cite{Z}), infers that the Wilf Conjecture is, in a sense, asymptotically true; the survey \cite{D} describes the state of the research on the Wilf Conjecture.

Despite this vibrant literature, the general case is still very elusive, and in fact, no bound for $F$ in function of $e$ and $n$, that holds true for all numerical semigroups, is known. In this work, we provide such a bound, by virtue of a bound for the smallest generator $g_1$ in function of $e$ and $n$.

\begin{thm}\label{Wilf}
	Let $g_1,\ldots,g_e$ be coprime positive integers larger than $1$, let $F$ be the Frobenius number, $n$ be the number of integers less than $F$ which are representable as a linear combination with coefficients in $\mathbb{N}$ of $g_1,\ldots,g_e$, and $\displaystyle q=\left \lceil \frac{F+1}{g_1} \right \rceil$. Then
	\begin{enumerate}
		\item $F+1 \le q e n;$
		\item $F+1 \le e n^2.$
	\end{enumerate}	
\end{thm}

Then, we provide a bound for the \emph{type} of the numerical semigroup $S=\langle g_1,\ldots,g_e \rangle$ in function of $e$ and $n$, and use this bound to give an alternative proof of the Wilf Conjecture when the numerical semigroup $S$ is almost-symmetric.

\section{Main result}
Let $\mathbb{Z}$ denote the set of integers, and $\mathbb{N}$ the set of non-negative integers. Given $e \ge 2$ and $g_1, \ldots, g_e \in \mathbb{N}$ such that $\gcd(g_1,\ldots,g_e)=1$, it is well-known that the set $$S=\langle g_1,\ldots,g_e \rangle = \{ a_1g_1+\ldots+a_eg_e \ | \ a_i \in \mathbb{N}\}$$
is a submonoid of $(\mathbb{N},+)$ such that the set $\mathbb{N} \setminus S$ is finite; a monoid $S$ satisfying this property is called a \emph{numerical semigroup} (see \cite{RG} for a detailed monograph on this algebraic structure). With the notation $S=\langle g_1,\ldots,g_e \rangle$ we will assume that $\{g_1,\ldots,g_e\}$ is a minimal generating system (which is unique for any numerical semigroup) for $S$, and we will thus say that $e$ is the \emph{embedding dimension} of $S$. We also denote by $F$ the \emph{Frobenius number} of $S$, that is, $F=\max \mathbb{Z} \setminus S$. Denote by $N(S)=S \cap [0,F]$ the set of elements of $S$ less than $F$ (called \emph{small elements}), and let $n=|N(S)|$.

Given an element $m \in S$, define the \emph{Ap\'ery set} of $S$ with respect to $s$ as $$Ap(S,m)=\{\omega \in S \ | \ \omega-m \not \in S\}.$$

Clearly $Ap(S,m)$ consists of the smallest elements of $S$ in every residual class modulo $m$, therefore $0 \in Ap(S,m)$, $\max Ap(S,m)=F+m$ and $|Ap(S,m)|=m$. 

Our first result is a bound for the smallest generator $g_1$ (often called the \emph{multiplicity}) of $S$, in function of $e$ and $n$. The main result is a direct corollary of this bound.

\begin{thm}\label{bound}
	Let $2 \le g_1< \ldots < g_e$ be coprime positive integers, and let $S=\langle g_1,\ldots,g_e \rangle$. Then
	
	$$g_1 \le (e-1)n+1.$$

\end{thm}

\begin{proof}
	Define the map $$\varphi: Ap(S,g_1)\setminus \{0\} \rightarrow \mathcal{P}(N(S) \times \{g_2,\ldots,g_e\}), \ \ \varphi(\omega)=\{(\omega-g_i,g_i) \ | \ \omega-g_i \in S\}.$$
	This map is well defined since, if $\omega \in Ap(S,g_1) \setminus \{0\}$, then $\omega \le F+g_1$, therefore $\omega-g_i \in S$ implies $\omega-g_i \in N(S)$. Moreover, for every $\omega \in Ap(S,g_1) \setminus \{0\}$, there exists a generator $g_i \le_S \omega$, and therefore $\varphi(\omega) \neq \emptyset$. Finally, for $\omega_1,\omega_2 \in Ap(S,g_1) \setminus \{0\}$, if $(s,g_i) \in \varphi(\omega_1) \cap \varphi(\omega_2)$ then $s=\omega_1-g_i=\omega_2-g_i$ and thus $\omega_1=\omega_2$: hence the sets $\varphi(\omega)$ are pairwise disjoint. Therefore the collection $\{\varphi(\omega)\}_{\omega \in Ap(S,g_1) \setminus \{0\}}$ is a partition of a subset of $N(S) \times \{g_2,\ldots,g_e\}$, and thus we conclude that $$g_1-1 = |Ap(S,g_1) \setminus \{0\}| \le \sum_{\omega \in Ap(S,g_1) \setminus \{0\}} |\varphi(\omega)| \le |N(S) \times \{g_2,\ldots,g_e\}| = n(e-1).$$ 
\end{proof}

\begin{proof}[Proof of Theorem 1]

By Theorem \ref{bound}, we know that $g_1 \le (e-1)n+1$, thus multiplying by $q$ and remembering that $n \ge 1$, we obtain $$F + 1 \le g_1q \le q(e-1)n + q  =qen-qn+q \le qen.$$

Finally, since by definition of $q$ we have $\{0,g_1,2g_1,\ldots,(q-1)g_1\} \subseteq S \cap [0,F] = N(S)$, we have $q \le n$, therefore $$F+1 \le q e n \le e n^2.$$
\end{proof}

For a numerical semigroup $S=\langle g_1,\ldots,g_e \rangle$, define the set of \emph{pseudo-Frobenius numbers} of $S$ as the set $$PF(S)=\{\omega \not \in S \ | \ \omega + s \in S \text{ for every } s \in S \setminus \{0\}\}.$$

The cardinality of $PF(S)$ is called the \emph{type} of $S$, denoted by $t$. Since for every $\omega \in PF(S)$ and $m \in S \setminus \{0\}$, $\omega+m \in Ap(S,m) \setminus\{0\}$, we have $t \le g_1-1$. Our next result is a bound for $t$ in function of $e$ and $n$.

\begin{thm}\label{typebound}
	Let $2 \le g_1<\ldots<g_e$ be positive coprime integers, let $S=\langle g_1,\ldots,g_e \rangle$, and define $\displaystyle q=\left \lceil \frac{F+1}{g_1}\right \rceil \ge 1$. Then
	
	$$t \le (e-2)[n-q+1]+2 \le (e-2)n+2.$$
\end{thm}

\begin{proof}
	Assume that there are two elements $f_1,f_2 \in PF(S)$ such that $f_1=\lambda_1g_2-g_1$ and $f_2=\lambda_2g_2-g_1$, with $\lambda_1 > \lambda_2$; then $s=f_1-f_2=(\lambda_1-\lambda_2)g_2 \in S$, yielding $f_2+s =f_1 \in S$, a contradiction. Then there is at most one element of the form $\lambda g_2-g_1$ in the set $PF(S)$; let $f_2$ be such an element (if it exists), and let $PF'(S)=PF(S) \setminus \{F,f_2\}$ (if $f_2$ does not exist, then take $PF'(S)=PF(S)$).
	
	Define the function $\varphi: PF'(S) \rightarrow \mathcal{P}(N(S) \times \{g_3,\ldots,g_e\})$ as $\varphi(f)=\{(s,g_i) \ | \ s=f+g_1-g_i \in S\}$. This function is well-defined since $(s,g_i) \in \varphi(f)$ is such that $s=f+g_1-g_i < f \le F$, $\varphi(f) \neq \emptyset$ (becuse, being $f \neq f_2$, $f+g_1$ cannot be of the form $Kg_2$, for some integer $K$), and clearly $\varphi(f) \cap \varphi(f')=\emptyset$ if $f \neq f'$, since $(s,g_i) \in \varphi(f) \cap \varphi(f')$ would imply $f=s+g_i-g_1=f'$. Therefore the collection $\{\varphi(f)\}_{f \in PF'(S)}$ is a partition of a subset of $N(S) \times \{g_3,\ldots,g_e\}$. Moreover, our choice of $q$ means that for $i=1,\ldots,q-1$, $ig_1 \in N(S)$, but if $(ig_1, g_i) \in \varphi(f)$ for some $f$ and $g_i$, then $f=ig_1+g_i-g_1 \in S$, which is impossible. Therefore for every $i=1,\ldots,q-1$ and $j=3,\ldots,e$, $(ig_1, g_j)$ cannot belong to any set $\varphi(f)$.
	Combining these facts, and remembering that $q \ge 1$, we obtain $$t-2  \le |PF'(S)| \le \sum_{f \in PF'(S)} |\varphi(f)|=|\bigcup_{f \in PF'(S)} \varphi(f)| \le (e-2)[n-q+1] \le (e-2)n.$$
	
\end{proof}

Let $S=\langle g_1,\ldots,g_e \rangle$ be a numerical semigroup. We say that $S$ is \emph{almost-symmetric} if, for every $x \not \in S$, either $F-x \in S$ or $\{x,F-x\} \subseteq PF(S)$. Partitioning the interval $[0,F]$ in couples $\{x,F-x\}$, it is simple to see that, for an almost-symmetric numerical semigroup, $2n+t=F+2$.
Then Theorem \ref{typebound} can be used to provide an alternative proof of Wilf's Conjecture for almost-symmetric numerical semigroups (see \cite{B} for the original proof).

\begin{cor}
	Almost symmetric numerical semigroups satisfy Wilf's conjecture.
\end{cor}

\begin{proof}
	Let $S$ be an almost symmetric numerical semigroup, $S \neq \{0,g_1,\rightarrow\}$ (in this case it is immediate to check that the Wilf Conjecture still holds). Then in Theorem \ref{typebound} we have $q \ge 2$, and thus $$t \le [e-2][n-1]+2= [e-2]n-e+4.$$
	
	By definition of almost symmetric numerical semigroup we have $2n+t=F+2$, hence assuming that $e\ge 4$ (we recall that Wilf's Conjecture is always true if $e \leq 3$) we have
	
	$$F+1 \le en-e+3 < en.$$

\end{proof}

\section*{Acknowledgements}
The authors would like to thank Alessio Sammartano and Francesco Strazzanti for their helpful feedback on the main results of this paper. We also thank Shalom Eliahou for discussing this work during the International Meeting on Numerical Semigroups held in Rome in June 2022.

Both authors were partially funded by the project “Propriet\`a locali e globali
di anelli e di variet\`a algebriche”-PIACERI 2020–22, Universit\`a degli Studi
di Catania.

{\bf Statements and Declarations.} On behalf of all authors, the corresponding
author states that there is no conflict of interest.

\end{document}